\documentclass{amsart}

\usepackage{amsmath,amssymb,amsthm}

\usepackage{url}

\theoremstyle{plain}
 \newtheorem{theorem}{Theorem}[section]

\theoremstyle{plain}
 \newtheorem{lemma}[theorem]{Lemma}

\theoremstyle{definition}

\theoremstyle{definition}

\begin{document}

\title{Multiplicative functions with $f(p+q-n_0) = f(p)+f(q)-f(n_0)$}

\author{Poo-Sung Park}
\email{pspark@kyungnam.ac.kr}
\address{Department of Mathematics Education\\ 
Kyungnam University\\
Changwon-si, 51767 \\
Republic of Korea}


\maketitle

\begin{abstract}
Let $n_0$ be 1 or 3. If a multiplicative function $f$ satisfies $f(p+q-n_0) = f(p)+f(q)-f(n_0)$ for all primes $p$ and $q$, then $f$ is the identity function $f(n)=n$ or a constant function $f(n)=1$.
\end{abstract}


\section{Introduction}

In 2016 Chen, Fang, Yuan, and Zheng showed that if a multiplicative function $f$ satisfies $f(p+q+n_0) = f(p)+f(q)+f(n_0)$ with $1 \le n_0 \le 10^6$ then $f$ is the identity function provided $f(p_0) \ne 0$ for some prime $p_0$ \cite{CFYZ}. This is a variation of Spiro's paper in 1992 in which she dealt multiplicative functions satisfying $f(p+q) = f(p)+f(q)$ \cite{Spiro}. She called the set of primes an \emph{additive uniqueness set} for multiplicative functions $f$ with $f(p_0) \ne 0$ for some prime $p_0$.

A natural question follows about $n_0$ being negative for Chen et al.'s paper. It is natural to consider the condition $f(p+q-n_0) = f(p)+f(q)-f(n_0)$ with $n_0 = 1, 2, 3$ because a multiplicative function is defined on positive integers.

The author already studied a multiplicative function satisfying $f(p+q-2) = f(p)+f(q)-f(2)$, which also yields that the set of numbers 1 less than primes is an additive uniqueness set for multiplicative functions \cite{Park}.

In this article we classify multiplicative functions satisfying $f(p+q-n_0) = f(p)+f(q)-f(n_0)$ with $n_0 = 1,3$. For consistency we state the classification for $n_0 = 2$ as well.

\begin{theorem}\label{thm:-1}
If a multiplicative function $f$ satisfies $f(p+q-1) = f(p)+f(q)-f(1)$ for all primes $p$ and $q$, then $f$ is the identity function $f(n)=n$ or a constant function $f(n)=1$.
\end{theorem}

\begin{theorem}[\cite{Park}]\label{thm:-2}
If a multiplicative function $f$ satisfies $f(p+q-2) = f(p)+f(q)-f(2)$ for all primes $p$ and $q$, then $f$ is the identity function $f(n)=n$, a constant function $f(n)=1$, or $f(n)=0$ for $n \ge 2$ unless $n$ is odd and squareful.
\end{theorem}

\begin{theorem}\label{thm:-3}
If a multiplicative function $f$ satisfies $f(p+q-3) = f(p)+f(q)-f(3)$ for all primes $p$ and $q$, then $f$ is the identity function $f(n)=n$ or a constant function $f(n)=1$.
\end{theorem}

Theorem \ref{thm:-2} for $n_0 = 2$ has one more option. We give a proof for Theorem \ref{thm:-3}. The proof of Theorem \ref{thm:-1} is similar and the proof of Theorem \ref{thm:-2} is given in \cite[\S 4]{Park}.

\section{Lemmas}

\begin{lemma}\label{lem:235711}
Assume a multiplicative function $f$ satisfies $f(p+q-3) = f(p) + f(q) - f(3)$ for all primes $p$ and $q$. Then, $f(n) = 1$ or $f(n) = n$ for $n = 2,3,5,7$, and $11$.
\end{lemma}

\begin{proof}
Note that $f(1)=1$ and the equalities
\begin{align*}
f(1) 
	&= f(2+2-3) = f(2)+f(2)-f(3), \\
f(7) 
	&= f(5+5-3) = f(5)+f(5)-f(3), \\
f(10) 
	&= f(2)\,f(5) \\
	&= f(11+2-3) = f(11)+f(2)-f(3), \\
f(11) 
	&= f(7+7-3) = f(7)+f(7)-f(3), \\
f(15) 
	&= f(3) \,f(5) \\
	&= f(11+7-3) = f(11)+f(7)-f(3).
\end{align*}

Let $a=f(2)$, $b=f(3)$, $c=f(5)$, $d=f(7)$, $e=f(11)$. Then,
\begin{align}
1 &= 2a-b \\
d &= 2c-b \\
ac &= e+a-b \\
e &= 2d-b \\
bc &= e+d-b.
\end{align}

The equation (3) becomes
\[
	ac = 4c - 7a + 4
\]
by the equations (1), (2), and (4). Also, the equation (5) becomes
\[
	2ac = 7c - 10a + 5.
\]
So, $c = 4a-3$ and we obtain an equation $a^2 -3a + 2 = 0$.

Thus, $a=1$ or $a=2$ and it follows that
\[
\begin{array}{lllll}
f(2) = 1, & f(3) = 1, & f(5) = 1, & f(7) = 1, & f(11) = 1; \\
f(2) = 2, & f(3) = 3, & f(5) = 5, & f(7) = 7, & f(11) = 11.
\end{array}
\]
\end{proof}

\begin{lemma}\label{lem:odd<10^10}
The results in Lemma \ref{lem:235711} can be extended up to $n$ odd and $n < 10^{10}$.
\end{lemma}

\begin{proof}
We use induction. Let $n$ be odd and $11 < n < 10^{10}$. 

If $n$ is prime, then $n = 6k-1$ or $n = 6k+1$. Suppose $n = 6k-1$. Note that
\[
	f(n+4) = f(6k+3) = f(n+7-3) = f(n) + f(7) - f(3).
\]
Since $6k+3$ can be factored into the product of two smaller integers, $f(6k+3) = 1$ or $f(6k+3) = 6k+3$ by induction hypothesis. Thus, $f(n) = 1$ or $f(n) = n$ when $n = 6k-1$ is prime.

Similarly, if $n$ is a prime of the form $6k+1$, then $f(n) = 1$ or $f(n) = n$ by
\[
	f(n+2) = f(6k+3) = f(n+5-3) = f(n) + f(5) - f(3).
\]

If $n$ is not a prime, $n$ is either a product of two relatively prime integers or a power of a prime. The first case is easy by the multiplicity of $f$. So the second case remains.

Now, assume that $n$ is a power of a prime with exponent $\ge 2$. Then, $n+3$ is even and can be written as a sum of two primes $p$ and $q$ with $5 \le p, q < n$ by the numerical verification of the Goldbach Conjecture up to $4 \times 10^{18}$ \cite{OeS-H-P}.

Then, since $f(n) = f(n+3-3) = f(p+q-3) = f(p) + f(q) - f(3)$, we obtain that $f(n) = 1$ or $f(n) = n$ by the induction hypothesis.
\end{proof}

Indeed, those can be extended up to $n \le 4 \times 10^{18} - 3$.

\begin{lemma}\label{lem:even<10^10}
The results in Lemma \ref{lem:235711} can be extended up to $n$ even and $n < 10^{10}$.
\end{lemma}

\begin{proof}
It is enough to investigate $f(2^r)$ with $r \le 33$. Note that $k \cdot 2^r + 1$ with $k < 2^r$ is called Proth number. If a Proth number is prime, it is called a Proth prime. It is verified that there exists an odd integer $k \le 4141$ such that $k \cdot 2^r + 1$ is a Proth prime for $1 \le r \le 1000$ in The On-Line Encyclopedia of Integer Sequences (OEIS, \url{https://oeis.org/A057778}), although the infinitude of Proth primes is not yet proved \cite{ProthSearch}.

Then, $k \cdot 2^r +1$ is an odd prime and
\[
	f(k)\,f(2^r) = f\big( (k \cdot 2^r + 1) + 2 - 3 \big) = f(k \cdot 2^r + 1) + f(2) - f(3).
\]
Thus, we are done by Lemmas \ref{lem:235711} and \ref{lem:odd<10^10}.
\end{proof}

If the Goldbach Conjecture and the infinitude of Proth primes for all exponents were proved, Theorem \ref{thm:-3} could be easily proved. But, neither of them has not yet been proved, so that we need other strategy. In the following lemma, $v_p(n)$ means the exponent of $p$ in the prime factorization of $n$ when $p$ is a prime and $n$ is a positive integer. The set $H$ was defined by Spiro and the numerical verification of the Goldbach Conjecture was up to $2\times10^{10}$ at that time. We would call the set $H$ in the lemma the \emph{Spiro set}.

\begin{lemma}\label{lem:m+q_in_H}
Let
\[
H = \{ n \,|\, v_p(n) \le 1 \text{ if } p>1000; v_p(n) \le \lfloor 9\log_p{10} \rfloor - 1 \text{ if } p<1000 \}.
\]
For any integer $m > 10^{10}$, there is an odd prime $q \le m-1$ such that $m+q \in H$. 
\end{lemma}

\begin{proof}
This lemma is the consequence of \cite[Lemma 2.4]{CFYZ} which follows the proof of \cite[Lemma 5]{Spiro}.
\end{proof}

%

\begin{lemma}[\cite{Cudakov, Estermann, vdCorput}]\label{lem:almost_every_integer}
Almost every even positive integer is expressible as the sum of two primes.
\end{lemma}

\begin{lemma}\label{lem:n_in_H}
The restricted function $f|_H$ is the identity function or a constant function on $H$. 
\end{lemma}

\begin{proof}
Assume $f(n) = n$ for $n = 2, 3, 5, 7, 11$.
If $n < 10^{10}$, then $f(n)=n$ from Lemmas \ref{lem:odd<10^10} and \ref{lem:even<10^10}. Let $n \in H$ with $n \ge 10^{10}$ and assume that $f(m)=m$ for all $m \in H$ with $m < n$. If $n$ is not a prime power, then $f(n) = f(a)f(b)$ with $(a,b)=1$ and $a,b > 1$. Since $f(a)=a$ and $f(b)=b$ by the induction hypothesis, $f(n)=n$. 

Now, if $n$ is a prime power, then $n$ is a prime by the definition of $H$. If $n = 6k-1$, then consider $n + 7 - 3 = 6k+3$. Since
\[
	f(n+4) = f(6k+3) = f(n+7-3) = f(n) + f(7) - f(3)
\]
and $6k+3$ can be factored into the product of two smaller integers, $f(n) = n$.

Similarly, if $n = 6k+1$, then
\[
	f(n+2) = f(6k+3) = f(n+5-3) = f(n) + f(5) - f(3)
\]
yields $f(n) = n$.

By the same reasoning, we can conclude that $f(n) = 1$ if $f(2) = f(3) = f(5) = f(7) = f(11) = 1$.

\end{proof}

\begin{lemma}[\protect{\cite[Lemma 7]{Spiro}}]\label{lem:H_n}
For any positive integer $n$, put
\[
H_n = 
\begin{cases}
\{ mn \,:\, m \in H, (m,n)=1\} & \text{if } 2 \mid n; \\
\{ 2mn \,:\, 2m \in H, (m,n)=1\} & \text{if } 2 \nmid n. \\
\end{cases}
\]
Then $H_n$ satisfies the following properties:
\begin{enumerate}
\item Every element of $H_n$ is even.
\item The set $H_n$ has positive lower density.
\end{enumerate}
\end{lemma}

\section{Proofs of Theorems}

Let us start to prove Theorem \ref{thm:-3}. Suppose that there exists $n$ for which $f(n) \ne n$. For $kn \in H_n$, we have that
\[
	f(kn) = f(k)\,f(n) = k\,f(n).
\]
If $f(kn) = kn$, then $f(n) = f(kn)/k = kn/k = n$, which contradicts. So $f(kn) \ne kn$ for every $kn \in H_n$.

But, if $kn+3$ with $k$ odd can be represented as a sum of two primes $p$ and $q$, then
\[
	f(kn) = f(p+q-3) = f(p) + f(q) - f(3) = p+q-3 = kn.
\]
Thus, this implies that there exist many counterexamples to the Goldbach Conjecture whose density is positive. But, this contradicts Lemma \ref{lem:almost_every_integer}. Therefore, $f(n) = n$ for all $n$.

We can prove Theorem \ref{thm:-1} in the similar way. First, we have that
\begin{align*}
f(3) 
	&= f(2+2-1) = f(2)+f(2)-f(1), \\
f(5)
	&= f(3+3-1) = f(3)+f(3)-f(1), \\
f(6) 
	&= f(2)\,f(3) \\
	&= f(5+2-1) = f(5) + f(2) - f(1).
\end{align*}

Let $a = f(2)$, $b = f(3)$, and $c = f(5)$. Then,
\[
b = 2a - 1, \qquad
c = 2b - 1, \qquad
ab = c + a - 1.
\]
Thus,
\[
	a(2a-1) = \big(2(2a-1)-1\big)+a-1
\]
and it becomes
\[
	a^2 - 3a + 2 = 0. 
\]
Hence, $a = 1$ or $a = 2$.

Next, we should check $f(2^r)$ as in Lemma \ref{lem:even<10^10}. We can use $k \cdot 2^r - 1$ instead of $k \cdot 2^r + 1$. The list of prime $k \cdot 2^r - 1$ with $0 \le r \le 10000$ is in OEIS (\url{https://oeis.org/A126717}). See also \cite{ProthSearch}.

%
%
%
%
%
%

\section*{Acknowledgment}
The author would like to thank the Korea Institute for Advanced Study(KIAS) for its support and hospitality.



\begin{thebibliography}{9}



\bibitem{CFYZ} Y.-G. Chen, J.-H. Fang, P. Yuan, and Y. Zheng, 
On multiplicative functions with $f(p + q + n_0)=f(p)+f(q)+f(n_0)$, 
\textit{J. Number Theory} \textbf{165} (2016), 270--289.

\bibitem{Cudakov} N. \v{C}udakov, 
On the Goldbach's problem, 
\textit{Dokl. Akad. Nauk SSSR} (1937), 331--334.

\bibitem{Estermann} T. Estermann, 
On the Goldbach's problem: Proof that almost all even positive integers are sums of two primes,
\textit{Proc. London Math. Soc. (2)} \textbf{44} (1938), 307--314.

\bibitem{OeS-H-P} T. Oliveira e Silva, S. Herzog, and S. Pardi, 
Empirical verification of the even the Goldbach conjecture and computation of prime gaps up to $4\cdot10^{18}$, 
\textit{Math. Comp.} \textbf{83} (2014) 2033--2060.

\bibitem{Park} P.-S. Park,
Additive uniqueness of $\mathrm{PRIMES}-1$ for multiplicative functions,
\textit{Int. J. Number Theory},
to appear.

\bibitem{ProthSearch} Proth Search Page, \url{http://www.prothsearch.com}

\bibitem{Spiro} C. A. Spiro, 
Additive uniqueness sets for arithmetic functions, 
\textit{J. Number Theory} \textbf{42} (1992), 232--246.

\bibitem{vdCorput} J. G. van der Corput,
Sur l'hypoth\`ese de the Goldbach,
\textit{Proc. Akad. Wet. Amsterdam} \textbf{41} (1938), 76--80.

\end{thebibliography}
\end{document}